\documentclass[11pt]{amsart}
\usepackage{amsthm}

\theoremstyle{plain}
\newtheorem{thm}{Theorem}[section]
\newtheorem{theorem}[thm]{Theorem}

\newtheorem{lemma}[thm]{Lemma}
\newtheorem{corollary}[thm]{Corollary}

\theoremstyle{definition}
\newtheorem{remark}[thm]{Remark}

\newtheorem{definition}[thm]{Definition}

\numberwithin{equation}{section}

 \title[Pseudo-isomorphisms of threefolds and Monge-Ampere of singular currents]{The Monge-Ampere operator of some singular (1,1) currents coming from pseudo-isomorphisms in dimension $3$}
 \author{Tuyen Trung Truong}
   \address{Department of Mathematics, University of Oslo, Blindern 0851 Oslo, Norway}
  \email{tuyentt@math.uio.no}
    \date{\today}

\begin{document}
\maketitle

\begin{abstract}
A wide and natural class of closed currents - which are differences of positive closed currents - can be constructed by pulling back smooth closed forms using rational maps. These currents are very singular in general, and hence defining intersections between them is challenging. In this paper, we use our previous results to investigate this question in the case  where the rational maps in question are pseudo-isomorphisms (i.e.  bimeromorphic maps which, along with their inverses, have no exceptional divisors) in dimension $3$. Our main result, to be described in a more concrete form later in the paper, is as follows. 

{\bf Theorem.} Let $X,Y$ be compact K\"ahler manifolds of dimension $3$, and $f:X\dashrightarrow Y$ be a pseudo-isomorphism. Let $\alpha _2,\alpha _3$ be smooth closed $(1,1)$ forms on $Y$, and $T_1$  a difference of two positive closed $(1,1)$ currents on $X$. Then, whether the intersection of the currents $T_1$, $f^*(\alpha _2)$ and $f^*(\alpha _3)$ satisfies a Bedford-Taylor's type monotone convergence depends only on the cohomology classes of $\alpha _2,\alpha _3$. 

Special attention is given to the case where $T_1=f^*(\alpha _1)$ where $\alpha _1$ is a smooth closed $(1,1)$ form on $Y$. It is then shown that satisfying the above mentioned Bedford-Taylor's type monotone convergence is asymmetric in $\alpha _1$, $\alpha _2$ and $\alpha _3$, but in contrast the resulting signed measure is symmetric in $\alpha _1$, $\alpha _2$ and $\alpha _3$. We relate this Bedford-Taylor's type monotone convergence to the least-negative intersection we defined previously. These results can be extended to the case where $\alpha _1$, $\alpha _2$, $\alpha _3$ are more singular. The use of global approximations (instead of local approximations) of singular currents and dynamics of pseudo-isomorphisms in dimension $3$ are essential in proving these results. At the end of the paper we discuss the intersection of currents which are differences of positive closed $(1,1)$ currents in general.   
\end{abstract}

\section{Introduction} 

Intersection of positive closed currents on complex manifolds is a challenging question which is very intensively studied, and which has many applications in different fields such as complex analysis, algebraic geometry and complex dynamics. In this paper, using the approach proposed in our previous papers \cite{truong, truong1, truong2},  we investigate the intersections of singular currents of the form $f^*(\alpha )$, where $\alpha$ is a smooth closed $(1,1)$ form and $f$ is a pseudo-isomorphism in dimension $3$. This class of currents is wide and natural, and can serve to test various definitions of intersection of positive closed currents. (We will show in Theorem \ref{Theorem2} that in general these currents are very singular.)  We recall that a pseudo-isomorphism is a bimeromorphic map $f:X\dashrightarrow Y$, so that both $f$ and $f^{-1}$ have no exceptional divisors. In birational geometry, when singularities are allowed, examples are "small contractions", including flips, which are essential in Mori's minimal model program. Denote by $I(f)$ the indeterminacy set of $f$. In case $\dim (X)=\dim (Y)$ is $3$, a result by Bedford and Kim \cite{bedford-kim} asserts that if $I(f)$ is not empty then its irreducible components are all curves.  

Besides going into detail of the approach which will be used in the current paper, we review briefly the relevance of other approaches to intersection of positive closed currents. There are two recent approaches - named superpotentials and density currents respectively - both proposed by Dinh and Sibony \cite{dinh-sibony, dinh-sibony1, dinh-sibony2}. These two methods are applied very successfully to many questions in complex dynamics and complex geometry, but it is unknown what answers these approaches will provide for the class of currents considered in this paper. There are also other approaches, all based basically on the idea of non-pluripolar intersection by Bedford and Taylor \cite{bedford-taylor2}. However, all methods based on non-pluripolar intersection usually provide answers which are not compatible with intersection in cohomology. The reason is in particular because they throw out - in one way or another - the parts of the currents which are singular, for example analytically singular or worse. For more about these methods, the readers are referred to \cite{truong} and references therein.   

Our approach in this paper is to use that developed in our previous papers, mentioned above. The main idea, for example in defining the intersection of closed currents $T_1,T_2,\ldots$ which are differences of two positive closed currents, is to first write $T_1=\Omega _1+dd^cR_1,~T_2=\Omega _2+dd^cR_2,\ldots $, where $\Omega _1,\Omega _2,\ldots $ are smooth closed forms. Then to find the set of cluster points of sequences $(\Omega _1+dd^cR_{1,n})\wedge (\Omega _2+dd^cR_{2,n})\wedge \ldots $, where $\{R_{1,n}\},\{R_{2,n}\},\ldots $ are sequences of smooth currents which approximate $R_1,R_2,\ldots $ and which satisfy some other good properties ensuring that the sequence is pre-compact. In the good case, it will be shown that there is only one cluster point, and hence it is natural to define the intersection $T_1\wedge T_2\wedge \ldots $ to be that limit. In the general case, where the cluster points are not unique, then we can take an upper bound of all cluster points, if that operator is available (for example, this is the case when the cluster points are signed measures, and has been worked out in \cite{truong}). As can be seen in these papers, and illustrated further in this paper, this approach has the advantage that it allows: 

- Symmetry in the role of the involved currents, and compatibility with intersection in cohomology.  

- The currents $T_1,T_2,\ldots $ to be not positive, and also allow the resulting intersection $T_1\wedge T_2\wedge \ldots $ to be not positive. (This point is at cases, while not always, very important, see the comments right after the statement of Theorem \ref{TheoremMain}.)

- The currents $T_1,T_2,\ldots $ to be very singular, for example having non-zero Lelong numbers on the same subvariety of $X$ of large dimension.

- The case where the intersection in cohomology of the concerned currents is negative.
 
\begin{remark}
We note that in the above description of our approach, the approximations $\{R_{1,n}\}, \{R_{2,n}\},\ldots $ are global, that is, defined on the whole space $X$. This is in contrast to a more common approach of using local approximations, for example in the non-pluripolar intersection mentioned above. In the local approximations, one works on small open sets $U$, on which we can write $T_i=dd^cS_i$ for some current $S_i$, and then approximate $S_i$ in $U$ by $\{S_{i,n}\}$. If $S$ is positive, then there are local approximations of $S_i$ so that $dd^cS_{i,n}\geq 0$ for all $n$, and hence when using such approximations we will only obtain positive closed currents as the results. In particular, such local approximations cannot be used to define self-intersection of positive closed $(1,1)$ currents whose self-intersection in cohomology is negative, for example the one mentioned in Subsection \ref{SubsectionLeastNegativeIntersection}. Also, such local approximations cannot be used to define for example a pullback, which is compatible with the pullback in cohomology, of curves in the indeterminacy set of the map $J_X$ in Theorem \ref{Theorem2}. 

\label{Remark0}\end{remark} 
 
  For intersection of positive closed $(1,1)$ currents in bidegree $(1,1)$ we have refined the above approach to restrict to regularisation sequences of the form $\Omega +dd^c\varphi _n$, where $\varphi _n$ is a decreasing sequence of quasi-psh functions. (Recall that a function $\varphi$ is psh if it is upper-semicontinuous and $dd^c \varphi $ is positive, and an upper-semicontinuous function $\varphi $ is quasi-psh if there is a smooth closed $(1,1)$ form $\Omega$ for which $\Omega +dd^c\varphi $ is positive. Note that these functions are allowed to have the value $-\infty$, but cannot be $-\infty$ identically.) This is to conform to that in general only monotone convergence is possible for intersection of positive closed $(1,1)$ currents. More precisely, we recall the following classical monotone convergence by Bedford and Taylor \cite{bedford-taylor}: 

\begin{theorem}
Let $u_1,\ldots ,u_k$ be locally {\bf bounded} psh functions on the unit disk $\{z\in \mathbb{C}^k:~|z|<1\}$. For each $j$, let $\{u_{j,n}\}_n$ be a sequence of continuous psh functions decreasing point wise to $u_j$. Then the sequence of measures $dd^cu_{1,n}\wedge \ldots \wedge dd^cu_{k,n}$ converges to a unique measure $\mu$, independent of the choice of the regularisation sequences $\{u_{j,n}\}_n$. We can therefore define $dd^cu_{1}\wedge \ldots \wedge dd^cu_{k}:= \mu$. 
\label{TheoremBedfordTaylor}\end{theorem} 
 
In particular, the intersection of positive closed $(1,1)$ currents with locally bounded potentials is symmetric. For example, in the case of $3$ such currents, we have $dd^cu_1\wedge (dd^c u_2\wedge dd^c u_3)=dd^c u_2\wedge (dd^c u_1\wedge dd^c u_3)$ and so on. This symmetry has been regarded as a sign of whether a definition is a good one. Of course, we do not expect this symmetry in general, if we want to define the intersection of currents using limits of smooth ones and stay in the class of currents or measures. For example, if we have two positive closed currents $T_1$ and $T_2$ so that $T_1\wedge T_2=0$ or more generally a smooth form, then $T_1\wedge (T_1\wedge T_2)$ is obviously defined, while there is in general no way to define $T_1\wedge T_1$  and hence $T_2\wedge (T_1\wedge T_1)$. On the other hand, in \cite{truong}, we showed that this symmetry is preserved if one is ready to go beyond the class of measures to that of submeasures.  
 
In this paper, we will deal with currents of the form $f^*(\alpha )$, where $f$ is a rational map and $\alpha$ is a smooth closed $(1,1)$ form on a compact K\"ahler manifold. Such a current is in general not positive, but is a difference of two positive closed $(1,1)$ currents. In fact, if we write $\alpha =\alpha ^+-\alpha ^-$ where $\alpha ^{\pm}$ are K\"ahler forms, then $f^*(\alpha ^{\pm})$ are positive closed $(1,1)$ currents, and $f^*(\alpha )=f^*(\alpha ^+)-f^*(\alpha ^-)$ is a difference of two positive closed currents. If we write $f^*(\alpha ^{\pm})=\Omega ^{\pm}+dd^cu^{\pm}$, where $\Omega ^{\pm}$ are smooth closed $(1,1)$ forms, then we can write $f^*(\alpha )=\Omega +dd^cu$, where $\Omega =\Omega ^+-\Omega ^-$ is a smooth closed $(1,1)$ form, and $u=u^{+}-u^{-}$ is a difference of two quasi-psh functions. For convenience, we will write $[f^*(\alpha ), \Omega ,u^+,u^-]$ for such a representation of $f^*(\alpha )$. In particular, if $f^*(\alpha )$ is positive, we have a representation of the form $[f^*(\alpha ),\Omega ,u, 0]$, and if we have another representation $[f^*(\alpha ), \widetilde{\Omega},\widetilde{u},0]$ then $dd^c(u-\widetilde{u})=\Omega -\widetilde{\Omega}$ is smooth, and hence $u-\widetilde{u}$ must be smooth. Thus we see that $u$ and $\widetilde{u}$ have the same type of singularities, and it does not matter too much which representation we use. In contrast, two different representations of the general form $[f^*(\alpha ),\Omega ,u^+,u^-]$ and $[f^*(\alpha ),\Omega ,\widetilde{u^+},\widetilde{u^-}]$ need not to have the same type of singularities. In fact, if we have one representation $[f^*(\alpha ), \Omega ,u^+,u^-]$ for $f^*(\alpha )$, then $[f^*(\alpha ), \Omega ,u^++v,u^-+v]$ is also a representation of $f^*(\alpha )$, where $v$ is any quasi-psh function.  If $v$ is very singular, it will destroy the good properties of $(u^+,u^-)$, which is what we do not want. Therefore, when dealing with non positive currents, we need to carefully specify which representation of the current we want to use. If possible, we should choose a representation which is in some sense "least singular". See further discussions in the last part of this paper.   

In the light of the above comment, we modify the above monotone regularisation of quasi-psh functions to monotone regularisation of a pair. Let $(u^+,u^{-})$ be a pair of two quasi-psh functions, defined on a K\"ahler manifold. A sequence of pairs of decreasing continuous quasi-psh functions $(\{u^+_{n}\},\{u^{-}_n\})$ is called a {\bf good approximation} for $(u^+,u^-)$ if: i) There is a K\"ahler form $\omega $ so that $\omega +dd^cu^+_n$ and $\omega +dd^cu^-_n$ are positive closed currents for all $n$, and ii)  $u^+=\lim _{n\rightarrow \infty}u^+_n$ and $u^-=\lim _{n\rightarrow\infty}u^{-}_n $. The existence of such good approximations was proven by Demailly, see Section 2 for more detail. 

Our first main result in this paper is the following, which shows that whether the intersection of currents of the form $f^*(\alpha )$, where $f$ is a fixed pseudo-isomorphism in dimension $3$ and $\alpha$ is a smooth closed $(1,1)$ form satisfying a Bedford-Taylor's type monotone convergence can be characterised cohomologically.  

\begin{theorem} Let $X,Y$ be compact K\"ahler manifolds of dimension $3$, and $f:X\dashrightarrow Y$ a pseudo-isomorphism. Let $\alpha _2,\alpha _3$ be smooth closed $(1,1)$ forms on $Y$, and $T_1$ a difference of two positive closed $(1,1)$ currents on $X$. 

1) Let $[T_1,\Omega _1,u_1^+,u_1^-]$ be a representation of $T_1$. Assume that there is a signed measure $\mu _{[T_1,\Omega _1,u_1^+,u_1^-],f^*(\alpha _2),f^*(\alpha _3)}$ so that for all good approximations $(\{u^+_{1,n}\},\{u^-_{1,n}\})$ of $(u_1^+,u_1^-)$, we have:
\begin{eqnarray*}
\lim _{n\rightarrow\infty}(\Omega _1+dd^c(u^+_{1,n}-u^-_{1,n}))\wedge f^*(\alpha _2)\wedge f^*(\alpha _3)=\mu _{[T_1,\Omega _1,u_1^+,u_1^-],f^*(\alpha _2),f^*(\alpha _3)}. 
\end{eqnarray*}

Then for a smooth closed $(1,1)$ form $\widetilde{\alpha _2}$  cohomologous to $\alpha _2$ and a smooth closed $(1,1)$ form $\widetilde{\alpha _3}$  cohomologous to $\alpha _3$, there is a signed measure $\mu _{[T_1,\Omega _1,u_1^+,u_1^-],f^*(\widetilde{\alpha _2}), f^*(\widetilde{\alpha _3})}$ so that for all good approximation $(\{u^+_{1,n}\},\{u^-_{1,n}\})$ of $u_1$, we have:
\begin{eqnarray*}
\lim _{n\rightarrow\infty}({\Omega _1}+dd^c({u^+_{1,n}}-{u^-_{1,n}}))\wedge f^*(\widetilde{\alpha _2})\wedge f^*(\widetilde{\alpha _3})=\mu _{[T_1,\Omega _1,u_1^+,u_1^-],f^*(\widetilde{\alpha _2}),f^*(\widetilde{\alpha _3})}. 
\end{eqnarray*}

2) Fix $\Omega _1$ a smooth closed $(1,1)$ form cohomologous to $T_1$. Assume that for every representation of $T_1$ of the form $[T_1, \Omega _1, u_1^+,u_1^-]$, where we allow $u_1^{\pm}$ to vary, there is a signed measure $\mu _{[T_1,\Omega _1,u_1^+,u_1^-],f^*(\alpha _2),f^*(\alpha _3)}$ so that for all good approximations $(\{u^+_{1,n}\},\{u^-_{1,n}\})$ of $(u_1^+,u_1^-)$, we have:
\begin{eqnarray*}
\lim _{n\rightarrow\infty}(\Omega _1+dd^c(u^+_{1,n}-u^-_{1,n}))\wedge f^*(\alpha _2)\wedge f^*(\alpha _3)=\mu _{[T_1,\Omega _1,u_1^+,u_1^-],f^*(\alpha _2),f^*(\alpha _3)}. 
\end{eqnarray*}
  
Then the measures $\mu _{[T_1,\Omega _1,u_1^+,u_1^-],f^*(\alpha _2),f^*(\alpha _3)}$ are independent of the choice of $u_1^{\pm}$ and hence we can write it as $\mu _{T_1,\Omega _1,f^*(\alpha _2),f^*(\alpha _3)}$.  Moreover,  for every representation $[T_1, \widetilde{\Omega _1},\widetilde{u_1^+}, \widetilde{u_1^-}]$, and  for all good approximation $(\{\widetilde{u^+_{1,n}}\},\{\widetilde{u^-_{1,n}}\})$ of $(\widetilde{u_1^+}, \widetilde{u_1^-})$, we have:
\begin{eqnarray*}
\lim _{n\rightarrow\infty}(\widetilde{\Omega _1}+dd^c(\widetilde{u^+_{1,n}}-\widetilde{u^-_{1,n}}))\wedge f^*(\alpha _2)\wedge f^*(\alpha _3)=\mu _{T_1,\Omega _1,f^*(\alpha _2),f^*(\alpha _3)}. 
\end{eqnarray*}
Hence, the signed measure $\mu _{T_1,\Omega _1,f^*(\alpha _2),f^*(\alpha _3)}$ is independent of the choice of $\Omega _1$, and we can write it as $\mu _{T_1,f^*(\alpha _2),f^*(\alpha _3)}$.
\label{TheoremMain}\end{theorem}

We note that in Theorem \ref{TheoremMain}, the intersection $f^*(\alpha _2)\wedge f^*(\alpha _3)$ is well-defined, since the currents $f^*(\alpha _2)$ and $f^*(\alpha _3)$ are smooth outside of curves in $I(f)$. Moreover, the intersection $f^*(\alpha _2)\wedge f^*(\alpha _3)$ itself satisfies a Bedford-Taylor's monotone convergence, see \cite{fornaess-sibony, demailly2}. Hence, having the existence of limits as in the statement of the theorem, it is reasonable to define the intersection of $T_1$, $f^*(\alpha _2) $ and $f^*(\alpha _3)$ to be that limit, in accordant to the classical cases (see Theorem \ref{TheoremBedfordTaylor} and comments right below it). 

We remark that in works by other authors, mentioned above, only intersection of positive closed currents is considered. This cannot be directly applied in the situation considered here in general. In fact, consider the special case where $T_i=f^*(\alpha _i)$ ($i=1,2,3$) for a pseudo-isomorphism $f:X\dashrightarrow Y$, where $\alpha _i$ is a smooth closed $(1,1)$ form on $Y$.  One natural way to reduce to the situation of intersection of positive closed currents is to write $\alpha _i=\alpha _i^+ -\alpha _i^{-}$ (for $i=1,2,3$), where $\alpha _i^{\pm}$ are positive closed smooth $(1,1)$ forms (i.e. K\"ahler forms), and then try to define $f^*(\alpha _1^{\pm})\wedge f^*(\alpha _2^{\pm})\wedge f^*(\alpha _3^{\pm})$ separately. However, as seen in Theorem \ref{Theorem2} below, such an approach cannot be successful, since in general $f^*(\alpha _1^{\pm})$, $f^*(\alpha _2^{\pm})$ and  $f^*(\alpha _3^{\pm})$ will all have non-zero Lelong numbers simultaneously along some curves in $I(f)$. 

The main idea for the proof of Theorem \ref{TheoremMain} is to first prove Theorem \ref{TheoremSufficientCondition} below. Since in general a positive closed smooth $(1,1)$ form $\alpha$ will never satisfy the cohomological condition (the NIC condition below) in Theorem \ref{TheoremSufficientCondition}, we see again that here it is essential to work with non-positive forms $\alpha$, and hence with non-positive currents $f^*(\alpha )$.   

In view of the above theorem, we define the following two notions of Bedford-Taylor's type monotone convergence. 

\begin{definition}
Let $X$ be a compact K\"ahler $3$-fold. Let $T_1,T_2,T_3$ be differences of positive closed $(1,1)$ currents on $X$, so that $T_1\wedge T_2$, $T_1\wedge T_3$ and $T_2\wedge T_3$ are well-defined in the sense of \cite{fornaess-sibony, demailly2}.  

{\bf Weak BTC} (for Weak Bedford-Taylor monotone convergence).  $T_1$, $T_2$ and $T_3$ are said to satisfy the Weak BTC condition if there is a representation $[T_1,\Omega _1,u_1^+,u_1^-]$ and a signed measure $\mu _{[T_1,\Omega _1,u_1^+,u_1^-],T _2,T_3}$ so that the following condition is satisfied. For all good approximations $(\{u^+_{1,n}\},\{u^-_{1,n}\})$ of $(u_1^+,u_1^-)$, we have:

\begin{eqnarray*}
\lim _{n\rightarrow\infty}(\Omega _1+dd^c(u^+_{1,n})-dd^c(u^-_{1,n}))\wedge T_2\wedge T_3=\mu _{[T_1,\Omega _1,u_1^+,u_1^-], T_2,T_3}. 
\end{eqnarray*}

{\bf BTC} (for Bedford-Taylor's monotone convergence).  $T_1$, $T_2$ and $T_3$ are said to satisfy the BTC condition if there is a signed measure $\mu _{T_1,T_2,T_3}$ so that for all representations $[T_1, \Omega _1,u_1^+,u_1^-]$ and for all good approximations $(\{u^+_{1,n}\},\{u^-_{1,n}\})$ of $(u_1^+,u_1^-)$, we have:
\begin{eqnarray*}
\lim _{n\rightarrow\infty}(\Omega _1+dd^c(u^+_{1,n}-u^-_{1,n}))\wedge T_2\wedge T_3=\mu _{T_1,T_2,T_3}. 
\end{eqnarray*}

\label{DefinitionBTC}\end{definition}

Theorem \ref{TheoremMain} says that when $T_i=f^*(\alpha _i)$ as above, both the BTC condition and the Weak BTC condition depend only on the cohomology classes of $\alpha _2$ and $\alpha _3$. The BTC condition is the strongest form of intersection satisfying a Bedford-Taylor's type monotone convergence between (non)-positive closed $(1,1)$ currents we can expect. Theorem \ref{Theorem2} below shows that even for this special case, the (Weak) BTC condition is not always satisfied. 

The next cohomological condition provides an easy-to-check  criterion for the BTC condition. 

\begin{definition}
{\bf NIC} (for Null intersection with indeterminacy curves). Let $f:X\dashrightarrow Y$ be a pseudo-isomorphism in dimension $3$. A cohomology class $\gamma\in H^{1,1}(X)$ is said to satisfy the NIC condition if for all curves $C\subset I(f^{-1})$ we have in cohomology: $\gamma .\{C\}=0$. 
\label{DefinitionNIC}\end{definition}

The next result characterises the BTC condition.  

\begin{theorem} Let $X,Y$ be compact K\"ahler manifolds of dimension $3$, and $f:X\dashrightarrow Y$ a pseudo-isomorphism. Let $\alpha _1,\alpha _2,\alpha _3$ be smooth closed $(1,1)$ forms on $Y$. Let $T_1$ be a difference of two positive closed $(1,1)$ currents on $X$. 

1)  There exists a signed measure, denoted by $T_1\wedge f^*(\alpha _2\wedge \alpha _3)$, so that for every representation $[T_1, \Omega _1, u_1^+,u_1^-]$ and for all good approximations $(\{u^+_{1,n}\},\{u^-_{1,n}\})$ of $(u_1^+,u_1^-)$, we have: 
\begin{eqnarray*}
\lim _{n\rightarrow\infty}(\Omega +dd^cu_{1,n}^+-dd^cu_{1,n}^{-})\wedge f^*(\alpha _2\wedge \alpha _3)=T_1\wedge f^*(\alpha _2\wedge \alpha _3).
\end{eqnarray*}

2) If $T_1,f^*(\alpha _2),f^*(\alpha _3)$ satisfy the BTC condition, then $\mu _{T_1,f^*(\alpha _2),f^*(\alpha _3)}=T_1\wedge f^*(\alpha _2\wedge \alpha _3)$. In particular, if $T_1=f^*(\alpha _1)$, then $\mu _{f^*(\alpha _1),f^*(\alpha _2),f^*(\alpha _3)}=f^*(\alpha _1\wedge \alpha _2\wedge \alpha _3)$ and we must have in cohomology $\{f^*(\alpha _1)\}.\{f^*(\alpha _2)\}.\{f^*(\alpha _3)\}=\{f^*(\alpha _1\wedge \alpha _2\wedge \alpha _3)\}$. 

3) $T_1$, $f^*(\alpha _2)$ and $f^*(\alpha _3)$  satisfy the BTC condition iff $f^*(\alpha _2)\wedge f^*(\alpha _3)=f^*(\alpha _2\wedge \alpha _3)$. The latter condition depends only on the cohomology classes of $\alpha _2$ and $\alpha _3$. 

In particular, if at least one of the cohomology classes $\{\alpha _2\}$ and $\{\alpha _3\}$ satisfies the NIC condition, then $T_1$, $f^*(\alpha _2)$ and $f^*(\alpha _3)$ satisfies the BTC condition. 

4) On the other hand, if $T_1$ is smooth then $T_1$, $f^*(\alpha _2)$ and $f^*(\alpha _3)$ always satisfy the Weak BTC condition for every $\alpha _2$ and $\alpha _3$. 
\label{TheoremSufficientCondition}\end{theorem}

As a consequence, we obtain the following result concerning the symmetry of the BTC condition. 

\begin{corollary}
Let $f:X\dashrightarrow Y$ be a pseudo-isomorphism in dimension $3$. Let $\alpha _1$, $\alpha _2$ and $\alpha _3$ be smooth closed $(1,1)$ forms on $Y$.

1) The BTC condition is asymmetric in arguments. More precisely, there are $\alpha _1$, $\alpha _2$ and $\alpha _3$ such that $f^*(\alpha _1)$, $f^*(\alpha _2)$ and $f^*(\alpha _3)$ satisfy the BTC condition, but $f^*(\alpha _2)$, $f^*(\alpha _1)$ and $f^*(\alpha _3)$ do not satisfy the BTC condition. 

2) On the other hand, the resulting measure is symmetric in arguments. More precisely, if $f^*(\alpha _1)$, $f^*(\alpha _2)$, $f^*(\alpha _3)$ satisfy the BTC condition, then $\mu _{f^*(\alpha _1), f^*(\alpha _2),f^*(\alpha _3)}$ is symmetric in $\alpha _1,\alpha _2,\alpha _3$.   
\label{Corollary1}\end{corollary}
\begin{proof}
1) Choose $\alpha _1$, $\alpha _2$, $\alpha _3$ be such  that $\alpha _2$ satisfies the NIC condition but $f^*(\alpha _1\wedge \alpha _3)\not= f^*(\alpha _1)\wedge f^*(\alpha _3)$. Then by part 3) of Theorem \ref{TheoremSufficientCondition}, we have the desired conclusion. 

2) By part 2 of Theorem \ref{TheoremSufficientCondition} we have $\mu _{f^*(\alpha _1),f^*(\alpha _2),f^*(\alpha _3)}=f^*(\alpha _1\wedge \alpha _2\wedge \alpha _3)$ and hence it is symmetric in $\alpha _1,\alpha _2$ and $\alpha _3$. 

\end{proof}

While the Weak BTC condition is more difficult to deal with, we show below that some constraints must be satisfied for it to hold. We also relate this condition to the least negative intersection of positive closed $(1,1)$ currents defined in the paper \cite{truong}. This is a sublinear operator and is symmetric in the arguments. A brief review of the least negative intersection will be provided in Section 2.

\begin{theorem}
Let $f:X\dashrightarrow Y$ be a pseudo-isomorphism of compact K\"ahler $3$-folds. Let $\alpha _2,\alpha _3$ be positive closed smooth $(1,1)$ forms on $Y$, and $T_1$ a difference of two positive closed $(1,1)$ currents on $X$. Assume that $T_1$, $f^*(\alpha _2)$ and $f^*(\alpha _3)$ satisfy the Weak BTC condition with respect to a representation $[T_1,\Omega _1,u_1^+,u_1^-]$ of $T_1$.

1) Let $\Omega _1^+$ and $\Omega _1^-$ be smooth closed $(1,1)$ forms on $X$ so that $T_1^{\pm}=\Omega _1^{\pm}+dd^cu_1^{\pm}$ are positive and $T_1=T_1^+-T_1^-$.  Then in cohomology we must have $\{T_1^{\pm}\}.\{f^*(\alpha _2)\}.\{f^*(\alpha _3)\}\geq 0$. Moreover, the least negative intersections $\Lambda (T_1^{\pm},f^*(\alpha _2), f^*(\alpha _3))$ are non-negative. 

In particular, if $T_1$ is positive, then in cohomology we must have  $\{T_1\}.\{f^*(\alpha _2)\}.\{f^*(\alpha _3)\}\geq 0$, and the least negative intersection $\Lambda (T_1,f^*(\alpha _2), f^*(\alpha _3))$ is non-negative. 

2) If $T_1$ is positive and $\{T_1\}.\{f^*(\alpha _2)\}.\{f^*(\alpha _3)\}=\{T_1\}.\{f^*(\alpha _2\wedge \alpha _3)\}$, then the least negative intersection $\Lambda (T_1,f^*(\alpha _2),f^*(\alpha _3))$ equals $T_1\wedge f^*(\alpha _2\wedge \alpha _3)$. 

Moreover, the triple $T_1,f^*(\alpha _2),f^*(\alpha _3)$ satisfies the Weak BTC condition with respect to  every representation of the form $[T_1,\Omega _1, u_1,0]$,  and  we have in this case $\mu _{[T_1,\Omega _1, u_1,0],f^*(\alpha _2),f^*(\alpha _3)}$ $=$ $T_1\wedge f^*(\alpha _2\wedge \alpha _3)$. 
\label{TheoremWeakBTCCondition}\end{theorem}
 We note that the condition $\{T_1\}.\{f^*(\alpha _2)\}.\{f^*(\alpha _3)\}=\{T_1\}.\{f^*(\alpha _2\wedge \alpha _3)\}$ is represented by a hypersurface in the set of $(\{T_1\},\{\alpha _2\},\{\alpha _3\})$ in $H^{1,1}(X)\times H^{1,1}(Y)\times H^{1,1}(Y)$, while the condition $f^*(\alpha _2\wedge \alpha _3)=f^*(\alpha _2)\wedge f^*(\alpha _3)$ is in general represented by a higher codimension subvariety. 
  
 Next we give some explicit calculations for a special but interesting map, which provide (counter-)examples to various aspects of the above main theorems. In particular, it is shown that the NIC condition is necessary in some cases for the (Weak) BTC condition to be satisfied.   
\begin{theorem}
Let $X$ be the blowup of $\mathbb{P}^3$ at $e_0=[1:0:0:0]$, $e_1=[0:1:0:0]$, $e_2=[0:0:1:0]$ and $e_3=[0:0:0:1]$. Let $J_X$ be the lift to $X$ of the birational map $J[x_0:x_1:x_2:x_3]=[1/x_0:1/x_1:1/x_2:1/x_3]$ of $\mathbb{P}^3$. Then $J_X=J_X^{-1}$ is a pseudo-automorphism, whose indeterminacy set is non-empty and consists of $6$ pairwise disjoint curves. 

1) If $\alpha $ is a K\"ahler form on $X$, then for all curve $C\subset I(J_X^{-1})$ we have in cohomology: $\{J_X^*(\alpha )\}.\{C\}<0$. Moreover, $J_X^*(\alpha )$ has non-zero Lelong numbers along every curve $C\subset I(J_X)=I(J_X^{-1})$.  

Likewise, there is a K\"ahler form $\alpha$ so that in cohomology we have $$J_X^*(\{\alpha \}).J_X^*(\{\alpha\}).J_X^*(\{\alpha \})<0.$$ 

2) There is a unique non-zero cohomology class $\eta _0\in H^{1,1}(X)$ so that any $\eta $ satisfying the NIC condition must be a scalar multiple of $\eta _0$. 
Moreover, $\eta _0$ is nef, and there is a positive closed current $T$ on $X$ whose cohomology class is $\eta _0$ and is smooth outside the exceptional divisors of the blowup $X\rightarrow \mathbb{P}^3$. 
 
3) Two smooth closed $(1,1)$ forms $\alpha _2,\alpha _3$ satisfy $J_X^*(\alpha _2)\wedge J_X^*(\alpha _3)=J_X^*(\alpha _2\wedge \alpha _3)$ iff for all curves $C\subset I(J_X)$ we have in cohomology $\{\alpha _2\}.\{C\}=0$ or  $\{\alpha _3\}.\{C\} =0$. 

Hence, the set of such pairs of cohomology classes $(\{\alpha _2\},\{\alpha _3\})\in H^{1,1}(X)\times H^{1,1}(X)$ is a non-empty union of linear subspaces, each of dimension $\geq 4$. 

There are many different such pairs where both $\{\alpha _2\}$ and $\{\alpha _3\}$ are nef. 

Moreover, when $\alpha _2=\alpha _3$, we have $J_X^*(\alpha _2)\wedge J_X^*(\alpha _3)=J_X^*(\alpha _2\wedge \alpha _3)$ if and only if $\alpha _2$ satisfies the NIC condition. 

4) Let $\alpha _1,\alpha _2,\alpha _3$ be smooth closed $(1,1)$ forms, where $\alpha _1$ is positive. If $J_X^*(\alpha _1)$, $J_X^*(\alpha _2)$ and $J_X^*(\alpha _3)$ satisfy  the Weak BTC condition, then for all curve $C\subset I(J_X)$ we have in cohomology: either $\alpha _{1}.C=0$, or $\alpha _{2}.C=0$, or $\alpha _{3}.C=0$. 

In particular, if $\alpha _1=\alpha _2=\alpha _3$, then $J_X^*(\alpha _1)$, $J_X^*(\alpha _2)$ and $J_X^*(\alpha _3)$ satisfy the Weak BTC condition if and only if $\alpha _1$ satisfies the NIC condition. 
\label{Theorem2}\end{theorem}

Generalisations of the above theorems to the case where $\alpha _1,\alpha _2,\alpha _3$ are allowed to be singular, which are more technical to state, will be given later.  

The plan of this paper is as follows. In Section 2 we give a brief overview of some background materials to be used in the paper. In Section 3 we prove some properties of pseudo-isomorphisms in dimension $3$ which are crucial for the proofs of main theorems. Then we prove the main theorems and their generalisations. At the end of the paper we have a general discussion on intersection of non-positive closed $(1,1)$ currents, and apply this to obtain a variant  of the BTC condition.    

{\bf  Acknowledgments.}  The current paper incorporates the author's two previous preprints \cite{truong3, truong4}. 

\section{Preliminaries} In this section we recall some background materials needed later in the paper.

\subsection{Pullback of currents by meromorphic maps} Let $X,Y$ be compact K\"ahler manifolds and let $f:X\dashrightarrow Y$ be a dominant meromorphic map. In the case $f$ is a holomorphic map, given any smooth closed form $\alpha$ on $Y$, we can pullback $f^*(\alpha )$ as a smooth closed form on $X$. When $f$ is not holomorphic, then it is no longer able to pullback $\alpha$ as a smooth form, but only as a current. Moreover, the computation is not as direct as in the case of holomorphic maps, and we need to use the graph of the map. Below is the detail.  

Let $\Gamma _f\subset X\times Y$ be the graph of $f$, and let $\pi _X,\pi _Y:X\times Y\rightarrow X,Y$ be the natural projections. Let $T$ be a positive closed $(p,p)$ current on $Y$. We can write $T=\Omega +dd^cR$ where $\Omega $ is a smooth closed $(p,p)$ form cohomologous to $T$.  We then define $f^*(\Omega )=(\pi _X)_*(\pi _Y^{-1}(\Omega )\wedge [\Gamma _f])$. If we can also define $\pi _Y^{-1}(dd^cR)\wedge [\Gamma _f]$, then we can define $f^*(T):=f^*(\Omega )+(\pi _X)_*(\pi _Y^{-1}(dd^cR)\wedge [\Gamma _f])$. One way to define $\pi _Y^{-1}(dd^cR)\wedge [\Gamma _f]$ is to first define $\pi _Y^{-1}(R)\wedge [\Gamma _f]$ and then take the $dd^c$ of the resulting current. We prefer that some continuity property, similar to Theorem \ref{TheoremBedfordTaylor}, should be satisfied. That is, for appropriate smooth approximations $\{R_n\}$ of $R$ (for example, the ones discussed in the next subsection), we should have $\lim _{n\rightarrow\infty} \pi _Y^{-1}(R_n)\wedge [\Gamma _f]$ exists, and we define $\pi _Y^{-1}(R)\wedge [\Gamma _f]$ to be that limit. 

Some cases where this strategy (or a modification of its) successes are listed here. For positive closed $(1,1)$ currents,  the quasi-potentials $R$ above are quasi-psh functions, and hence we can define $\pi _Y^{-1}(R)\wedge [\Gamma _f]$ by simply restricting $\pi _Y^{-1}(R)$ to the open dense subset of $\Gamma _f$ where the map $f$ is holomorphic and then extending by $0$, see M\'eo \cite{meo}. In this case, and in several other cases considered in the literature, the pullback currents are also positive. There are yet certain other cases where we can define pullback of currents by a modification of the above strategy, but the result of pulling back a positive closed current is no longer positive, see \cite{truong1, truong2} and the proof of Theorem \ref{Theorem2}. 

We can apply the above ideas to the case where $T$ is a difference of two positive closed currents, or more generally a DSH current. Recall \cite{dinh-sibony3}, that a DSH $(p,p)$  current is of the form $R_1-R_2$, where $R_i$ are negative currents so that $dd^cR_i=T_i^+-T_i^{-}$ for some positive closed currents $T_i^{\pm}$.  This type of currents is needed in approximating positive closed $(p,p)$ currents when $p>1$, see the next subsection for more detail. 

\subsection{Regularisation of positive closed currents} As would be now clearer to the readers, having good approximations of positive closed currents is very useful when we want to define some operators $\mathcal{G}(T_1,T_2\ldots )$ on them (such as intersection of positive closed currents or pullback of positive closed currents by meromorphic maps). In the good case, where the currents under consideration have good properties and the operations have an "obvious" definition (for example when we consider positive closed $(1,1)$ currents with locally bounded potentials for the intersection problem, and smooth closed forms for the problem of pulling back by meromorphic maps), there are theorems (such as Bedford-Taylor's monotone convergence, see Theorem \ref{TheoremBedfordTaylor}) which assure that if we consider the same operation but now on good approximations $T_{i,n}$ of the concerned currents, then the sequence $\mathcal{G}(T_{1,n},T_{2,n},\ldots )$ will converge to the "obvious" answer, and hence the "obvious" answer should be the desired one.  This idea can also been adapted also to the cases where no "obvious" answers exist, as the ones considered in this paper. 

From the above discussion, it is important to single out some classes of currents with good properties which can serve as good approximations of positive closed $(p,p)$ currents on a compact K\"ahler manifold $(X,\omega )$. Here we consider $p$ only in the range $1\leq p\leq \dim (X)-1$. It turns out that the two cases where $p=1$ and $p>1$ must be treated separately. 

When $p=1$, it is a classical result that locally a psh function $\varphi$ can be approximated by a sequence of decreasing smooth psh functions $\varphi _n$. The global case is harder, and was treated by Demailly \cite{demailly1}. Given a positive closed $(1,1)$ current $T$, we can write $T=\Omega +dd^cu$, where $u$ is a quasi-psh function. We can arrange that $u\leq 0$. The Lelong number of $T$ at a point $x\in X$ is defined by 
\begin{eqnarray*}
\nu (T,x)=\nu (u,x):=\liminf _{z\rightarrow x}\frac{u(x)}{\log |z-x|}.
\end{eqnarray*}
The Lelong number is a non-negative number. In case $T=[W]$ is the current of integration of an analytic subvariety $W$ of pure dimension of $X$, then it is a classical result that $\nu ([W],x)$ is the multiplicity of $W$ at $x$. If $V\subset X$ is an analytic subvariety, then we define $\nu (T,V)=\inf _{x\in V}\nu (T,x)$. It follows from the definition that if $\nu (T,V)>0$, then 
\begin{eqnarray*}
\lim _{x\rightarrow V}u(x)=-\infty. 
\end{eqnarray*}
 This simple fact will be used later. 

We say that a quasi-psh function has analytic singularities if locally it is of the form $\log (|g_1(z)|^2+\ldots +|g_m(z)|^2)+\phi (z)$, where $g_1(z),\ldots ,g_m(z)$ are analytic and $\phi$ is a smooth function. Demailly proved two fundamental approximation results. 

\begin{theorem}
1) {Approximation Theorem 1.} There is a constant $C>0$ and a sequence of decreasing smooth quasi-psh functions $\{u_n\}$ so that $\{u_n\}$ converges to $u$ and $C\omega + dd^cu_n$ is positive for all $n$. 

2) {Approximation Theorem 2.}  There exist a decreasing sequence $\{\epsilon _n\} $ converging to $0$ and a sequence of quasi-psh functions $\{u_n\}$ with analytic singularities decreasing to $u$, so that $\Omega + dd^cu_n+\epsilon _n\omega $ is positive for all $n$ and the Lelong numbers of $\{u_n\}$ increase uniformly to the Lelong number of $u$ at every point in $X$. 
\label{TheoremDemailly}\end{theorem}

When $p>1$, it is in general  impossible to write a positive closed $(p,p)$ current $T$ as $T=\Omega +dd^cR$, where $\Omega$ is a smooth closed current and $R$ a negative current. However, we can write as such where $R$ is a $DSH$ $(p-1,p-1)$ current. Dinh and Sibony \cite{dinh-sibony3} proved an analog of Demailly's Approximation theorem 1 for this case, where we replace smooth quasi-psh functions $u_n$ by smooth DSH $(p-1,p-1)$ $R_n$ currents. Note however that in this case we cannot require that $R_n$ is negative or decreases to $R$, even when $R$ is a negative current. 

\subsection{Least-negative intersection of positive closed $(1,1)$ currents}\label{SubsectionLeastNegativeIntersection}

Here we recall briefly the least-negative intersection of $k$ positive closed $(1,1)$ currents on a compact K\"ahler manifold $X$ of dimension $k$, defined in the paper \cite{truong}. 

Let $T_i=\Omega _i+dd^cu_i$ be positive closed $(1,1)$ currents on $X$ ($i=1,2,\ldots ,k$). We would like to define an intersection product between these currents.  In the good case (for example, under the assumptions in Theorem \ref{TheoremBedfordTaylor}), the intersection is a positive measure and is symmetric in $T_1,\ldots ,T_k$, and satisfies a Bedford-Taylor's type monotone convergence. In the general case, however, we cannot expect all of these properties to be satisfied. In fact, under these requirements, the mass of the measure must be the same as the intersection of the cohomology classes of the involved currents. Hence, the measure (when it exists) cannot be positive if we have in cohomology $\{T_1\}\ldots \{T_k\}<0$. 

The least-negative intersection is designed to preserve as much of these properties as possible. The result we obtain is as close to a positive measure as could be, symmetric in the involved currents,  and the idea of using monotone approximation in Theorem \ref{TheoremBedfordTaylor} is incorporated in the definition of the least negative intersection, and hence it is also compatible with intersection in cohomology. The precise definition is as follows. 

For any good approximation $\{u_{i,n}\}$ of $u_i$ (for $i=1,\ldots ,k$), the currents $\Omega _i+dd^cu_{i,n}$ are smooth,  and the sequence of signed measures $(\Omega _1+dd^cu_{1,n})\wedge \ldots \wedge (\Omega _k+dd^cu_{k,n})$ is pre-compact in the space of signed measures, here the topology is by weak convergence on both the positive and negative parts of these signed measures. We define $\mathcal{G}$ to be the set of cluster points of all such sequences of signed measures we obtain when we run all over possible good approximations. If $\mu \in \mathcal{G}$, then the mass of $\mu$, which is the value $\mu (1)$ we obtain when applying $\mu$ to the constant function $1$, is the constant $\{T_1\}\ldots \{T_k\}$. 

Let $\overline{\mathcal{G}}$ be the closure of $\mathcal{G}$ in the space of signed measures. The set $\overline{\mathcal{G}}$ is in general not compact. We will consider the  subset $\mathcal{G}_0$ of signed measures closest to positive measures. To define this subset, we first define the negative norm $||\mu ||_{neg}$ of a signed measure $\mu$. We define: $||\mu ||_{neg}=\inf \{\mu ^{-}(1)$, here $\mu ^{-}$ runs on all over possible decompositions $\mu =\mu ^+-\mu ^-$ where both $\mu ^{\pm}$ are positive measures$\}$. This norm in some sense represents how close $\mu$ is to a positive measure. Indeed, $\mu$ is a positive measure iff $||\mu||_{neg}=0$. Then we define 
\begin{eqnarray*}
\kappa _{\mathcal{G}} = \min _{\mu \in \overline{\mathcal{G}}}||\mu ||_{neg}. 
\end{eqnarray*}
Then we define $\mathcal{G}_0:=\{\mu \in \overline{\mathcal{G}}:~ ||\mu ||_{neg}=\kappa _{\mathcal{G}}\}$. This is a compact set in the space of signed measures.

The least negative intersection $\Lambda (T_1,\ldots ,T_k)$ of $T_1,\ldots ,T_k$ is now defined as a sublinear operator on the space of continuous functions $\varphi$ on $X$, given by the formula: 
\begin{eqnarray*}
\Lambda (T_1,\ldots ,T_k) (\varphi )=\sup _{\mu \in \mathcal{G}_0}\mu (\varphi ).
\end{eqnarray*}
 
We say that  such a least negative intersection is non-negative if for all non-negative continuous functions $\varphi$ we have $\Lambda (T_1,\ldots ,T_k)(\varphi )\geq 0$. In the general case, non-negative least negative intersection can be regarded as the closest to what we have in the classical cases. It is shown in \cite{truong}, for $T_1=T_2=[L]$ the current of integration on a line $L\subset X=\mathbb{P}^2$, that for all continuous functions $\varphi$: 
\begin{eqnarray*}
\Lambda ([L],[L]) (\varphi ) =\sup _{x\in L}\varphi (x),
\end{eqnarray*}
and hence it is positive. Thus, we have a definition of the self-intersection of a line, which has not been done using other methods. 

More interestingly, in \cite{truong}, the following example is also given. Let $X=$ the blowup of $\mathbb{P}^2$ at a point $p$, and let $E\subset X$ be the exceptional divisor. Then the least negative intersection $\Lambda ([E],[E])$ acts on continuous functions $\varphi$ by the following formula: 
\begin{eqnarray*}
\Lambda ([E],[E])(\varphi )=\sup _{x\in E}-\varphi (x).
\end{eqnarray*}
Note that this is compatible with the fact that, in cohomology, the self-intersection of $E$ is $-1$. 

\section{Proofs of main theorems and generalisations}

We fix in this section compact K\"ahler 3-folds $X,Y $ and a pseudo-isomorphism map $f:X\dashrightarrow Y$. We fix also a resolution of singularity $Z$ of $\Gamma _f$, and let $\pi , h:Z\rightarrow X,Y$ the composition of $Z\rightarrow \Gamma _f$ and $\pi _X,\pi _Y:\Gamma _f\rightarrow X,Y$. We can choose $\pi $ and $h$ so that $\pi =\pi _m\circ \pi _{m-1}\circ \ldots \circ \pi _1$, where each $\pi _j$ is a blowup at a point or a smooth curve, and moreover the images by $\pi $ and $h$ of the exception divisors of $\pi$ are contained in $I(f)$ and $I(f^{-1})$ respectively.  

The next lemma, on dynamics of pseudo-isomorphisms in dimension $3$, is the key for the proofs of main results.  

\begin{lemma}
Let $T_1=T_1^{+}-T_1^{-}$ and $T_2=T_2^{+}-T_2^{-}$, where $T_i^{\pm}$ are positive closed $(1,1)$ currents on $Y$. Assume moreover that $T_{1}^{\pm} $ are smooth outside a curve. Then

1) $f^*(T_1\wedge T_2)$ and $f^*(T_1)\wedge f^*(T_2)$ are well-defined. 

2) The difference $f^*(T_1)\wedge f^*(T_2)-f^*(T_1\wedge T_2)$ has support in $I(f)$ and depends only on the cohomology classes of $T_1$ and $T_2$. 

3) Moreover,  if $\{T_1\}$ or $\{T_2\}$ satisfies the NIC condition, then $f^*(T_1)\wedge f^*(T_2)-f^*(T_1\wedge T_2)=0$.   
    
\label{Lemma1}\end{lemma}
\begin{proof}
1) Since $f$ is a pseudo-isomorphism in dimension $3$, by assumptions on $T_1^{\pm}$ it follows that $f^*(T_1^{\pm})$ are positive closed $(1,1)$ currents which are smooth outside some curves. Therefore, by dimensional reason we have that $f^*(T_1^{\pm})\wedge f^*(T_2^{\pm})$ are well-defined, and hence so is $f^*(T_1)\wedge f^*(T_2)$. 

 Since $S=T_1\wedge T_2$ is a positive closed $(2,2)$ current and $f$ is a pseudo-isomorphism in dimension $3$, it follows from \cite{truong1, truong2} that $f^*(T_1\wedge T_2)$ is well-defined. Moreover, if we take $\{S_n\}$ be any sequence of smooth approximations of $S$ constructed by Dinh and Sibony (mentioned in the previous section), then $\lim _{n\rightarrow\infty}f^*(S_n)$ exists and is the same as $f^*(T_1\wedge T_2)$. 
 
 2) We first prove the assertion for the case where $T_1$ and $T_2$ are both smooth. Then
 \begin{eqnarray*}
 f^*(T_1)=\pi _*h^*(T_1),~f^*(T_2)=\pi _*h^*(T_2),~f^*(T_1\wedge T_2)=\pi _*(h^*(T_1)\wedge h^*(T_2)).
 \end{eqnarray*}
Here $\pi _*=(\pi _m)_*\circ \ldots \circ (\pi _1)_*$. Define $\alpha =h^*(T_1)$ and $\beta =h^*(T_2)$, which are also smooth closed $(1,1)$ forms. 

First, we consider the blowup $\pi _1$. There are two cases to consider.

Case 1: $\pi _1$ is the blowup of one point. Then it follows from dimensional reason that the two $(2,2)$ currents $(\pi _1)_*(\alpha )\wedge (\pi _1)_*(\beta )$ and $(\pi _1)_*(\alpha \wedge \beta )$, which coincide outside the centre of the blowup $\pi _1$, must be the same. 

Case 2: $\pi _1$ is the blowup at a smooth curve $D_1$. Let $F_1$ (which is isomorphic to $\mathbb{P}^1$) be a fibre of the restriction of the blowup $\pi _1$ to the exceptional divisor $\pi _1^{-1}(D_1)$. From the results in \cite{truong5} we have that 
\begin{equation}
(\pi  _1)_*((\pi _1)^*(\pi _1)_*(\alpha )\wedge \beta )-(\pi _1)_*(\alpha )\wedge (\pi _1)_*(\beta )=\{\alpha .F_1\}\{\beta .F_1\}[D_1].
\label{Equation1}\end{equation}
Here $\{\alpha .F_1\}$ and $\{\beta .F_1\}$ are intersections in cohomology, and hence depending only on the cohomology classes of $\alpha ,\beta$, and $[D_1]$ is the current of integration over the curve $D_1$. Hence $(\pi _1)_*(\alpha )\wedge (\pi _1)_*(\beta )-(\pi _1)_*(\alpha \wedge \beta )$ depends only on the cohomology classes of $T_1$ and $T_2$,  if we can show that  the currents $(\pi  _1)_*((\pi _1)^*(\pi _1)_*(\alpha )\wedge \beta )$ and $(\pi _1)_*(\alpha )\wedge (\pi _1)_*(\beta )$ are the same. If we consider a smooth approximation $S_n$ of $(\pi _1)_*(\alpha )$, then we have
\begin{eqnarray*}
(\pi  _1)_*((\pi _1)^*(\pi _1)_*(\alpha )\wedge \beta )&=&\lim _{n\rightarrow\infty}(\pi _1)_*(\pi _1^*(S_n)\wedge \beta )\\
&=&\lim _{n\rightarrow\infty}S_n\wedge (\pi _1)_*(\beta )=(\pi _1)_*(\alpha )\wedge (\pi _1)_*(\beta ).
\end{eqnarray*}
The second equality follows from the projection formula, while the third equality follows from the fact that $(\pi _1)_*(\alpha )$ and $(\pi _1)_*(\beta )$ are differences of positive closed $(1,1)$ currents smooth outside curves, and that we are working in dimension $3$. 

Having shown the desired result for the blowup $\pi _1$, we can proceed to $\pi _2$  and other blowups in the same manner to obtain the desired result for the case where $T_1$ and $T_2$ are smooth. We only need to note that when we proceed to $\pi _2$, the currents $(\pi _1)_*(\alpha )$ and $(\pi _1)_*(\beta )$ are not smooth, but only smooth outside of a curve. Therefore, before applying the result in \cite{truong5}, we must first approximate $(\pi _1)_*(\alpha )$ and $(\pi _1)_*(\beta )$ by smooth closed $(1,1)$ forms, and use that when passing to the limit we obtain the desired results. 

For the general case where $T_1$ and $T_2$ are not smooth, we can use smooth approximations as above and proceed similarly. 

3) We can assume that $\{T_1\}$ satisfies the NIC condition. From the proof of 2), it follows that $f^*(T_1\wedge T_2)=f^*(T_1)\wedge f^*(T_2)$ if for all $j=1,\ldots ,m$ we have in cohomology $(\pi _m\circ \ldots \circ \pi _j)_*(h^*\{T_1\}).\{F_{j}\}=0$. By the projection formula, the same conclusion follows if $h^*\{T_1\}.(\pi _m\circ \ldots \circ \pi _j)^*\{F_j\}=0$ for all $j$. Note that the cohomology class of $(\pi _m\circ \ldots \circ \pi _j)^*\{F_j\}$ is generated by that of curves contained in the exceptional divisors of $\pi$. Therefore, by the projection formula again, the same conclusion follows if $\{T_1\}.h_*\{D\}=0$ for all curves $D$ contained in the exceptional divisors of $\pi$. By our assumption on the resolution of singularities $Z$ of the graph $\Gamma _f$, for any such $D$ the image $h(D)$ is contained in $I(f^{-1})$. Hence, we conclude that if $\{T_1\}.\{C\}=0$ for all  curves $C\subset I(f^{-1})$, then $f^*(T_1\wedge T_2)=f^*(T_1)\wedge f^*(T_2)$.   
\end{proof}

\subsection{Proof of Theorem \ref{TheoremSufficientCondition}}

1) We need to show that for all smooth functions $\varphi$ on $X$ and all good approximation $T_n$ of $T$  then
\begin{eqnarray*}
\lim _{n\rightarrow \infty}\int _X\varphi T_n \wedge f^*(\alpha _2\wedge \alpha _3) 
\end{eqnarray*}
exists. Note that for each $n$ then the integral equals
\begin{eqnarray*}
\int _Xf_*(\varphi T_n )\wedge \alpha _2\wedge \alpha _3. 
\end{eqnarray*}
Hence, it suffices to show that $\lim _{n\rightarrow\infty}f_*(\varphi T_n)$ exists. Note that $\varphi T$ is a DSH $(1,1)$ current, and $\varphi T_n$ is a good approximation of $\varphi T$, in the sense of Dinh and Sibony. Then the claim that  $\lim _{n\rightarrow\infty}f_*(\varphi T_n)$ exists follows from Theorem 1 in \cite{truong2} (see also \cite{truong1}), and the limit is denoted by $f_*(\varphi T)$ which is also a DSH current. 

Now consider the case where $T_1=f^*(\alpha _1)$, where $\alpha _1$ is a smooth closed $(1,1)$ form. In this case we need to show moreover that
\begin{eqnarray*}
\lim _{n\rightarrow \infty}\int _Xf_*(\varphi f^*(\alpha _1)) \wedge \alpha _2\wedge \alpha _3 =\int _{X} \varphi f^*(\alpha _1\wedge \alpha _2\wedge \alpha _3).
\end{eqnarray*} 
To this end, we can assume that $\varphi$, $\alpha _1$, $\alpha _2$ and $\alpha _3$ are all positive. First, we show that the positive measure $f_*(\varphi f^*(\alpha _1)) \wedge \alpha _2\wedge \alpha _3$ has no mass on proper subsets of $Y$. In fact, assume that $\varphi \leq C$. Then $f_*(\varphi f^*(\alpha _1)) \wedge \alpha _2\wedge \alpha _3$ is bounded from above by the measure $Cf_*(f^*(\alpha _1))\wedge \alpha _2\wedge \alpha _3$. Now, by Theorem 1 in \cite{truong2}, it follows that $f_*(f^*(\alpha _1))=\alpha _1$, and hence $f_*(\varphi f^*(\alpha _1)) \wedge \alpha _2\wedge \alpha _3$ is bounded from above by the smooth measure $C\alpha _1\wedge \alpha _2\wedge \alpha _3$, and hence has no mass on any proper subset of $Y$. Hence, $f_*(\varphi f^*(\alpha _1)) \wedge \alpha _2\wedge \alpha _3$ is the extension by $0$ of the its restriction to $Y\backslash I(f^{-1})$. On this latter set, we see easily that $f_*(\varphi f^*(\alpha _1)) \wedge \alpha _2\wedge \alpha _3=(f_*)(\varphi )\alpha _1\wedge \alpha _2\wedge \alpha _3$. From this the proof is completed.   

2) Assume that  $T_1,f^*(\alpha _2),f^*(\alpha _3)$ satisfy the BTC condition. By Lemma \ref{Lemma1}, we can write $f^*(\alpha _2)\wedge f^*(\alpha _3)=f^*(\alpha _2\wedge \alpha _3)+\sum _{C\subset I(f)}\lambda _{C,\alpha _2,\alpha _3}[C]$, where the sum is on all irreducible curves $C\subset I(f)$, and $\lambda _{C,\alpha _2,\alpha _3}$ depend only on the cohomology classes of $\alpha _2$ and $\alpha _3$. 

Consider any representation $[T_1,\Omega _1,u_1^+,u_1^-]$ and a good approximation $(\{u_{1,n}^+,u_{1,n}^{-}\})$ of $(u_{1}^+,u_1^-)$. By the assumption we have that for all smooth functions $\varphi $
\begin{eqnarray*}
\mu _{T_1,f^*(\alpha _2),f^*(\alpha _3)}(\varphi )&=&\lim _{n\rightarrow\infty}\int _X\varphi (\Omega _1+dd^cu_{1,n}^+-dd^cu_{1,n}^-)\wedge f^*(\alpha _2)\wedge f^*(\alpha _3)\\
&=&\lim _{n\rightarrow\infty}\int _X\varphi (\Omega _1+dd^cu_{1,n}^+-dd^cu_{1,n}^-)\wedge (f^*(\alpha _2\wedge \alpha _3)+\sum _{C\subset I(f)}\lambda _{C,\alpha _2,\alpha _3}[C]).
\end{eqnarray*}
 
By part 1, we have that 
\begin{eqnarray*}
\lim _{n\rightarrow\infty}\int _X\varphi (\Omega _1+dd^cu_{1,n}^+-dd^cu_{1,n}^-)\wedge f^*(\alpha _2\wedge \alpha _3)=\int _Xf_*(\varphi T_1)\wedge \alpha _2\wedge \alpha _3. 
\end{eqnarray*}

So it is sufficient to show that there is a representation $[T_1,\Omega _1,u_1^+,u_1^-]$ and a good approximation $(\{u_{1,n}^+,u_{1,n}^{-}\})$ of $(u_{1}^+,u_1^-)$ so that 
\begin{eqnarray*}
\lim _{n\rightarrow\infty}\int _X\varphi (\Omega _1+dd^cu_{1,n}^+-dd^cu_{1,n}^-)\wedge \sum _{C\subset I(f)}\lambda _{C,\alpha _2,\alpha _3}[C]=0.
\end{eqnarray*}

We now use that there is a representation $[T_1,\Omega _1, u_1^+,u_1^-]$ of $T_1$ so that $\lim _{x\rightarrow I(f)}u_1(x)^{\pm}=-\infty$. A simple way is to consider a blowup $p:X'\rightarrow X$ so that the preimages of all curves in $I(f)$ are divisors. Let $\omega '$ be any K\"ahler form on $X'$, then $p_*(\omega ')$ is a positive closed $(1,1)$ current with non-zero Lelong numbers along every curve in $I(f)$. We can also arrange that $p_*(\omega ')$ is smooth outside $I(f)$, by choosing appropriate blowups. Then  any quasi-potential $v$ of $p_*(\omega ')$ will have the property that $\lim _{x\rightarrow I(f)}v(x)=-\infty$. Moreover, $v$ is smooth outside $I(f)$. If we replace one representation $[T_1,\Omega _1, u_1^+,u_1^-]$ by $[T_1,\Omega _1, u_1^++v,u_1^-+v]$, using the fact that quasi-psh functions are bounded from above, and the properties of Lelong numbers, then we have the claim about the representation. [For better and more general constructions, see M\'eo \cite{meo} for projective spaces, and Vigny \cite{vigny} for compact K\"ahler manifolds.]

Now we consider the following sequence of quasi-psh functions $v_{1,n}^{\pm}=\max \{u_1^{\pm},-n\}$, with $u_1^{\pm}$ having the properties in the previous paragraph. For each $n$, there is an open neighbourhood of $I(f)$ on which $v_{1,n}^{\pm}$ are  the constant $-n$. In particular, $dd^c v_{1,n}^{\pm}=0$ in a neighbourhood of $I(f)$ and hence 
\begin{eqnarray*}
\lim _{n\rightarrow\infty} \int _X\varphi (\Omega _1+dd^cv_{1,n}^+-dd^cv_{1,n}^-)\wedge \sum _{C\subset I(f)}\lambda _{C,\alpha _2,\alpha _3}[C]=\int _X\varphi \Omega _1\wedge  \sum _{C\subset I(f)}\lambda _{C,\alpha _2,\alpha _3}[C].
\end{eqnarray*}

While $v_{1,n}^{\pm}$ are not smooth, we can construct a good approximation $(u_{1,n}^+,u_{1,n}^-)$ of $(u_1^+,u_1^-)$ such that (see Lemma 3.7 in \cite{truong})
\begin{eqnarray*}
\lim _{n\rightarrow\infty} \int _X\varphi dd^cu_{1,n}^{\pm}\wedge \sum _{C\subset I(f)}\lambda _{C,\alpha _2,\alpha _3}[C]=\lim _{n\rightarrow\infty} \int _X\varphi dd^cv_{1,n}^{\pm}\wedge \sum _{C\subset I(f)}\lambda _{C,\alpha _2,\alpha _3}[C],
\end{eqnarray*}
for all $\varphi$. Hence, from the BTC condition we find that: 
\begin{eqnarray*}
\mu _{T_1,f^*(\alpha _2),f^*(\alpha _3)}(\varphi )-\int _Xf_*(\varphi T_1)\wedge \alpha _2\wedge \alpha _3= \int _X\varphi \Omega _1\wedge  \sum _{C\subset I(f)}\lambda _{C,\alpha _2,\alpha _3}[C],
\end{eqnarray*}
for all choice of a smooth representative $\Omega _1$ of $T_1$. 

We will now conclude the proof by showing that $\lambda _{C,\alpha _2,\alpha _3}=0$ for all curves $C\subset I(f)$. Assume otherwise that there is a curve $C_0\subset I(f)$ so that $\lambda _{C_0,\alpha _2,\alpha _3}\not= 0$, we will obtain a contradiction as follows. We choose then a smooth point $x_0\in C$ and a small open neighbourhood $W$ of $x_0$ in $X$ which does not intersect other curves in $I(f)$. Choose $v$ a smooth function on $X$ so that $dd^c v$ is a K\"ahler form in $W$. Choose $\varphi $ a non-negative smooth function which has support in $W$ and which is $1$ near $x_0$. Then from the above we have, by choosing the two smooth representatives $\Omega _1$ and $\Omega _1+dd^c v$:
\begin{eqnarray*}
0= \int _X\varphi dd^cv\wedge  \sum _{C\subset I(f)}\lambda _{C,\alpha _2,\alpha _3}[C]= \lambda _{C,\alpha _2,\alpha _3}\int _X \varphi dd^cv\wedge [C_0],
\end{eqnarray*}
 which is the desired contradiction. 
 
 3) That the BTC condition is equivalent to that $f^*(\alpha _2)\wedge f^*(\alpha _3)=f^*(\alpha _2\wedge \alpha _3)$ can be seen from the proof of 2). That the condition $f^*(\alpha _2)\wedge f^*(\alpha _3)=f^*(\alpha _2\wedge \alpha _3)$ depends only on the cohomology classes of $\alpha _2,\alpha _3$ is shown in Lemma \ref{Lemma1}. By Lemma \ref{Lemma1} again, if $\alpha _2$ or $\alpha _3$ satisfies the NIC condition, then $f^*(\alpha _2)\wedge f^*(\alpha _3)=f^*(\alpha _2\wedge \alpha _3)$, and hence the BTC condition is satisfied. 
 
 4) If $T_1$ is smooth, then we have a representation $[T_1,T_1,0,0]$, and the assertion is clear. 
 
 \subsection{Proof of Theorem \ref{TheoremMain}}

1) The assertion follows from the following claim. 

Claim. Let $[T_1,\Omega _1, u_1^+,u_1^-]$ be a representation of $T_1$, and let $\{u_{1,n}^+,u_{1,n}^{-}\}$ be a good approximation of $(u_1^+,u_1^-)$. Assume that $\lim _{n\rightarrow \infty}(\Omega +dd^cu_{1,n}^+-dd^cu_{1,n}^{-})\wedge f^*(\alpha _2)\wedge f^*(\alpha _3)=\mu$. Let $\widetilde{\alpha _2}$ be a smooth closed $(1,1)$ form cohomologous to $\alpha _2$. Then 
\begin{eqnarray*}
\lim _{n\rightarrow \infty}(\Omega +dd^cu_{1,n}^+-dd^cu_{1,n}^{-})\wedge f^*(\widetilde{\alpha _2})\wedge f^*(\alpha _3)=\mu +T_1\wedge f^*((\widetilde{\alpha _2}-\alpha _2)\wedge \alpha _3).
\end{eqnarray*}
A similar result holds for $\alpha _3$ in the place of $\alpha _2$. 

Proof of Claim. Since the cohomology class of $\widetilde{\alpha _2}-\alpha _2$ is $0$, it trivially satisfies the NIC condition. Therefore, by Theorem \ref{TheoremMain} the triple $T_1,f^*(\widetilde{\alpha _2})-f^*(\alpha _2), f^*(\alpha _3)$ satisfies the BTC condition. In particular
\begin{eqnarray*}
\lim _{n\rightarrow \infty}(\Omega +dd^cu_{1,n}^+-dd^cu_{1,n}^{-})\wedge f^*(\widetilde{\alpha _2}-\alpha _2)\wedge f^*(\alpha _3)=T_1\wedge f^*((\widetilde{\alpha _2}-\alpha _2)\wedge \alpha _3).
\end{eqnarray*}

From this the claim follows. (Q.E.D.)

2) If $[T_1,\Omega _1,u_1^+,u_1^-]$ is a representation of $T_1$, then $[T_1,\Omega _1,u_1^++v,u_1^-+v]$ is also a representation, where $v$ is any quasi-psh function. We first show the following: 

Claim 1. Let $v$ be a quasi-psh function so that $\lim _{x\rightarrow I(f)}v(x)=-\infty$. Under the same assumptions as in part 2 of  Theorem \ref{Theorem2}, we have $\mu _{[T_1,\Omega _1,u_1^+,u_1^-], f^*(\alpha _2),f^*(\alpha _3)}=\mu _{[T_1,\Omega _1,u_1^++v, u_1^-+v],f^*(\alpha _2),f^*(\alpha _3)}$. 

Proof of Claim 1. Let $(u_{1,n}^+,u_{1,n}^-)$ be a good approximation of $(u_1^+,u_1^-)$. Then, as in the proof of part 2 of Theorem \ref{TheoremSufficientCondition}, we can find a good approximation $\{v_{n}\}$ of $v$ so that 
\begin{eqnarray*}
\lim _{n\rightarrow\infty} dd^cv_{n}\wedge [C]=\lim _{n\rightarrow\infty}dd^c \max\{v,-n\}\wedge [C]=0,
\end{eqnarray*}
for all curves $C\subset I(f)$. Since $(u_{1,n}^+ + v_{n},u_{1,n}^-+v_{n} )$ is a good approximation of $(u_1^++v,u_1^-+v)$, the Claim follows. (Q.E.D.)

Now we continue the proof of the theorem. We show next that for any representation $[T_1,\Omega _1,u_1^+,u_1^{-}]$ then 

Claim 2. 
\begin{eqnarray*}
\mu _{[T_1,\Omega _1, u_1^+,u_1^-],\alpha _2,\alpha _3}=T_1\wedge f^*(\alpha _2\wedge \alpha _3) + \Omega _1\wedge (f^*(\alpha _2)\wedge f^*(\alpha _3)-f^*(\alpha _2\wedge \alpha _3)). 
\end{eqnarray*}
Proof of Claim 2. In fact, using a good approximation $(u_{1,n}^+,u_{1,n}^-)$ of $(u_1^++v,u_1^-+v)$, where $v$ is as in the statement of Claim 1, so that
\begin{eqnarray*}
\lim _{n\rightarrow\infty} dd^cu_{1,n}^{\pm}\wedge [C]=\lim _{n\rightarrow\infty}dd^c \max\{u_1^{\pm},-n\}\wedge [C]=0,
\end{eqnarray*} 
for all curves $C\subset I(f)$, we obtain Claim 2. (Q.E.D.) 

Now if $\widetilde{\Omega _1}$ is another smooth closed $(1,1)$ form cohomologous to $T_1$, we have that $\widetilde{\Omega _1}-\Omega _1=dd^c u$ for some smooth function $u$. Then $[T_1,\widetilde{\Omega _1}, u_1^+, u_1^-+u]$ is another representation for $T_1$. If $(u_{1,n}^+,u_{1,n}^-)$ is a good approximation for $(u_1^+,u_1^-)$, then $(u_{1,n}^+,u_{1,n}^-+u)$ is a good approximation for $(u_1^+,u_1^-+u)$. Therefore, we have by definition: 
\begin{eqnarray*}
\mu _{T_1,\widetilde{\Omega _1},f^*(\alpha _2),f^*(\alpha _3)}&=&\lim _{n\rightarrow\infty}(\widetilde{\Omega _1}+dd^c(u_{1,n}^+)-dd^c(u_{1,n}^-+u))\wedge f^*(\alpha _2)\wedge \alpha _3\\
&=& \lim _{n\rightarrow\infty}({\Omega _1}+dd^c(u_{1,n}^+)-dd^cu_{1,n}^-)\wedge f^*(\alpha _2)\wedge \alpha _3\\
&=&\mu _{T_1.\Omega _1,f^*(\alpha _2),f^*(\alpha _3)}. 
\end{eqnarray*}
Hence, the proof of Theorem \ref{TheoremMain} is completed. 

{\bf Remark.} We can either use the same argument as in the proof of part 2 of Theorem \ref{TheoremSufficientCondition} or apply that theorem directly, to obtain that under the assumptions of part 2 of Theorem \ref{TheoremMain}, then $\mu _{T_1,f^*(\alpha _2),f^*(\alpha _3)}=T_1\wedge f^*(\alpha _2\wedge \alpha _3)$. 

\subsection{Proof of Theorem \ref{TheoremWeakBTCCondition}}

1) By definition of the Weak BTC condition, there is a signed measure $\mu _{[T_1,\Omega _1,u_{1}^+,u_1^-],f^*(\alpha _2),f^*(\alpha _3)}$ such that for all good approximations $(\{u_{1,n}^+\}, \{u_{1,n}^-\})$ of $(u_1^+,u_1^-)$ we have:
\begin{eqnarray*}
\lim _{n\rightarrow\infty}(\Omega _1+dd^cu_{1,n}^+-dd^cu_{1,n}^-)\wedge f^*(\alpha _2)\wedge f^*(\alpha _3)=\mu _{[T_1,\Omega _1,u_{1}^+,u_1^-],f^*(\alpha _2),f^*(\alpha _3)}.
\end{eqnarray*}

We write $f^*(\alpha _2)\wedge f^*(\alpha _3)=f^*(\alpha _2\wedge \alpha _3)+\sum _{C\subset I(f)}\lambda _C[C]$, where $C$ runs all over curves in $I(f)$. As in the proof of Theorem \ref{TheoremSufficientCondition}, we then have
\begin{eqnarray*}
\lim _{n\rightarrow\infty} (\Omega _1+dd^cu_{1,n}^+-dd^cu_{1,n}^-)\wedge \sum _{C\subset I(f)}\lambda _C[C] =\mu _{[T_1,\Omega _1,u_{1}^+,u_1^-],f^*(\alpha _2),f^*(\alpha _3)}-T_1\wedge f^*(\alpha _2\wedge \alpha _3). 
\end{eqnarray*}

We now prove that if $C_0\subset I(f)$ so that $\lambda _{C_0}\not= 0$, then both $u_{1}^{\pm}$ have zero Lelong number at $C_0$. In fact, assume for example that $u_1^+$ has non-zero Lelong number at the curve $C_0$. We choose $\{u_{1,n}^{-}\}$ to be any good approximation of $u_1^-$. We now choose $v$ to be a smooth function on $X$. Then the sequence of (non-smooth) quasi-psh functions $v_n^+=\max \{u_1^++v,-n\}-v$ decreases to $u_1^+$ and satisfies:
\begin{eqnarray*}
\lim _{n\rightarrow\infty}dd^cv_n^+\wedge [C_0] = -dd^cv\wedge [C_0]. 
\end{eqnarray*}
Hence, using Lemma 3.7 in \cite{truong} as in the proof of part 2 of Theorem \ref{TheoremSufficientCondition}, we obtain a contradiction. Hence we must have that both $u_1^{\pm}$ have zero Lelong number on curves $C\subset I(f)$ with $\lambda _C\not= 0$. Therefore, by using Demailly's Approximation Theorem 2 as in the proof of Theorem \ref{TheoremSufficientCondition}, we have that $\{T_1^{\pm}\}.\{f^*(\alpha _2)\}.\{f^*(\alpha _3)\}\geq 0$. Moreover, by using Theorem 3.1 in \cite{truong} we have that both least negative intersection products $\Lambda (T_1^{\pm},\alpha _2,\alpha _3)$ are non-negative.  

2) When $T_1$ is a positive current,  then with the above representation $[T_1,\Omega _1,u_1^+,u_1^-]$ we have $u_1=u_1^+-u_1^-$ is also a quasi-psh function. Moreover, we also have that the Lelong number of $u_1$ (which is bounded from above by that of $u_1^+$) at all curves $C\subset I(f)$ so that $\lambda _C\not= 0$ is $0$. Hence, the same argument as above shows that $\{T_1\}.\{f^*(\alpha _2)\}.\{f^*(\alpha _3)\}\geq 0$ and the least negative intersection $\Lambda (T_1,f^*(\alpha _2),f^*(\alpha _3))$ is non-negative. 

To finish the proof, we show that if moreover  $T_1$ is positive and $\{T_1\}.\{f^*(\alpha _2)\}.\{f^*(\alpha _3)\}=\{T_1\}.\{f^*(\alpha _2\wedge \alpha _3)\}$, then $\Lambda (T_1,f^*(\alpha _2),f^*(\alpha _3))=T_1\wedge f^*(\alpha _2\wedge \alpha _3)$. By the above results, we have that $\Lambda (T_1,f^*(\alpha _2),f^*(\alpha _3))$ is non-negative, it is bounded from below by  $T_1\wedge f^*(\alpha _2\wedge \alpha _3)$, and it has the same mass as $T_1\wedge f^*(\alpha _2\wedge \alpha _3)$. Moreover, $\Lambda (T_1,f^*(\alpha _2),f^*(\alpha _3))$ equals $T_1\wedge f^*(\alpha _2\wedge \alpha _3)$ on $X\backslash I(f)$, and $T_1\wedge f^*(\alpha _2\wedge \alpha _3)$ has no mass on $I(f)$ (looking at $f_*(T_1)\wedge \alpha _2\wedge \alpha _3$ and argue as above). Therefore, we must have $\Lambda (T_1,f^*(\alpha _2),f^*(\alpha _3))=T_1\wedge f^*(\alpha _2\wedge \alpha _3)$  as desired.

The claim about Weak BTC condition with respect to the representations of the form $[T_1,\Omega _1,u_1,0]$ follows easily from the above discussion. 
  
\subsection{Proof of Theorem \ref{Theorem2}}

Before the proof of Theorem \ref{Theorem2}, let us mention in passing that by simple calculations it is evident that for the map $J_X$ even the pullback of $dd^cu$, where $u$ is  a smooth function, is quite singular. In fact, if we work in the affine chart $x_0=1$, then for a function $u(x_1,x_2,x_3)$  
\begin{eqnarray*}
J^*dd^cu=dd^cu\circ J=dd^cu(\frac{1}{x_1},\frac{1}{x_2},\frac{1}{x_3})
\end{eqnarray*}
will have singular of the order $\frac{1}{x_j^2\overline{x_k^2}}$ near the line $\{x_j=x_k=0\}$. 

Part 1 of Theorem \ref{Theorem2} provides precise information on the pullback of smooth closed $(1,1)$ forms by $J_X$. Now we proceed with its proofs. We start with some properties of $J_X$ which were proven in \cite{truong1}. 

Let $E_0,E_1,E_2,E_3\subset X$ be the exceptional divisors corresponding to the centre of blowups $e_0,e_1,e_2,e_3$. Let $H$ be a generic hyperplane in $\mathbb{P}^3$ and $H^2$ a generic line in $\mathbb{P}^2$. Let $L_i\subset E_i=\mathbb{P}^2$ be a generic line, for $i=0,1,2,3$. Then $H,E_0,E_1,E_2,E_3$ are a basis for $H^{1,1}(X)$ and $H^2,L_0,L_1,L_2,L_3$ are a basis for $H^{2,2}(X)$. The intersection in cohomology on  $X$ is given by
\begin{eqnarray*}
H.H^2&=&1,\\
 H.E_i&=&0,\\
 E_i.E_j&=&-\delta _{i,j}L_i,\\
 E_i.L_j&=&-\delta _{i,j}. 
\end{eqnarray*}
Here $\delta _{i,j}$ is 1 if $i=j$, and is $0$ otherwise. 

Let $C_{i,j}\subset \mathbb{P}^3$ be the line going through $e_i$ and $e_j$, and let $\widetilde{C_{i,j}}\subset X$ be the strict transform of $C_{i,j}$. Then the cohomology $\{\widetilde{C_{i,j}\}}$ is $H^2-L_i-L_j$. The curves $C_{i,j}$ are pairwise disjoint. The map $J_X=J_X^{-1}$ is pseudo-automorphic, and $I(J_X)=\bigcup _{i,j}\widetilde{C_{i,j}}$. 
 
The pullback in cohomology of $J_X$ on $H^{1,1}$ is as follows: 
\begin{eqnarray*}
J_X^*(H)&=&3H-2\sum E_i,\\
J_X^*(E_j)&=&H+E_j-\sum E_i.
\end{eqnarray*}

By duality, the pullback of $J_X$ on $H^{2,2}$ is as follows: 
\begin{eqnarray*}
J_X^*(H^2)&=&3H^2-\sum L_i,\\
J_X^*(L_j)&=&2H^2+L_j-\sum L_i.
\end{eqnarray*}

The pullback of $J_X$ on positive closed $(2,2)$ currents is quite interesting, we have a loss of positivity. For example, we obtain $J_X^*[\widetilde{C_{0,1}}]=-[\widetilde{C_{2,3}}]$ on currents of integration. This equality has been obtained by showing that if $S_n$ is a good approximation of $\widetilde{C_{0,1}}$, as mentioned in Section 2, then $\lim _{n\rightarrow\infty}J_X^*(S_n)=-[\widetilde{C_{2,3}}]$. In particular, we have in cohomology $(J_X)^*\{\widetilde{C_{0,1}}\}=-\{\widetilde{C_{2,3}}\}$, which can also be computed using the above description of the pullback of $J_X$ on the cohmology groups. 

1) We show for example that if $\alpha $ is a K\"ahler form on $X$ then in cohomology $\{J_X^*(\alpha )\}.\{\widetilde{C_{0,1}}\}<0$. In fact, we have
\begin{eqnarray*}
\{J_X^*(\alpha )\}.\{\widetilde{C_{0,1}}\}=\{\alpha \}.(J_X)_*\{\widetilde{C_{0,1}}\}=-\{\alpha \}.\{\widetilde{C_{2,3}}\}<0. 
\end{eqnarray*}
 
Let $\alpha $ be a K\"ahler form. If $J_X^*(\alpha )$ had zero Lelong number on, say, $\widetilde{C_{0,1}}$, then from the fact that $J_X^*(\alpha )$ is smooth outside $\bigcup _{i,j}\widetilde{C_{i,j}}$ and that $\widetilde{C_{i,j}}$ are pairwise disjoint, we would have by Demailly's Approximation Theorem 2 that there is a sequence of closed $(1,1)$ forms $S_n$ which converges to $J_X^*(\alpha )$ and smooth in a neighbourhood of $\widetilde{C_{0,1}}$ and a sequence of positive numbers $\epsilon _n$ decreasing to $0$ so that $S_n+\epsilon _n\alpha $ are positive for all $n$. Then in cohomology $\{S_n+\epsilon _n\alpha \}.\{\widetilde{C_{0,1}}\}\geq 0$ for all $n$. Taking the limit we would get $\{J_X^*(\alpha )\}.\{\widetilde{C_{0,1}}\}\geq 0$, which is a contradiction. Therefore, $J_X^*(\alpha )$ must have non-zero Lelong number on $\widetilde{C_{0,1}}$, as wanted. 

We note that $J_X^*(H)^3=(3H-2\sum E_i)^3=27-4\times 8=-5<0$. Since for small enough positive numbers $\epsilon _0,\ldots ,\epsilon _3$ the cohomology class $H-\sum \epsilon _iE_i$ is a K\"ahler class, it follows that there is a K\"ahler form $\alpha $ so that in cohomology $\{J_X^*(\alpha )\}^3<0$. 

2) The uniqueness of the class $\eta _0$, up to a multiplicative constant, is easy to deduce from the intersection in cohomology on $X$. We can choose indeed $\eta _0=2H-\sum _iE_i$. 

The fact that $\eta _0$ is nef can be implied from Kleiman's criterion for nefness and the easy description of the cone of curves of $X$. For another way to see this claim, see the next paragraph. 

If we take $W\subset \mathbb{P}^3$ to be a smooth surface of degree $2$ which contains $e_0,e_1,e_2,e_3$, then the strict transform of $W$ in $X$ has the cohomology class equal to $\eta _0$. If we use the method by M\'eo \cite{meo}, taking the average on all such $W$'s, we obtain a positive closed $(1,1)$ current $S_0$ on $\mathbb{P}^3$ which has mass $2$ and Lelong number $1$ at $e_0,e_1,e_2,e_3$, and is smooth outside these $4$ points. Hence, if $S$ is the strict transform of $S_0$ in $X$, then $S$ is smooth outside $\bigcup _iE_i$, and $S$ has the cohomology class equal to  $\eta _0$. 

3)   If $\pi :Z\rightarrow X$ is the blowup at $\bigcup _{i,j}\widetilde{C_{i,j}}$, then the lifting map $J_Z$ is an automorphism of $Z$. Hence, the map $h=J_X\circ \pi =\pi \circ J_Z$ is holomorphic. Therefore, we can use this data $(Z,\pi ,h)$ in the proof of part 2 of Lemma \ref{Lemma1}. 

If we take $E_{i,j}\subset Z$ be the exceptional divisor corresponding to the centre $\widetilde{C_{i,j}}$ of the blowup $\pi$, and let $L_{i,j}$ be a generic fibre of the restriction $\pi : E_{i,j}\rightarrow \widetilde{C_{i,j}}$, then $h(\widetilde{C_{i,j}})=\widetilde{C_{3-i,3-j}}$. Moreover, if $D_{i,j}\subset E_{i,j}$ is a section of the restriction $\pi : E_{i,j}\rightarrow \widetilde{C_{i,j}}$, then in cohomology: 
\begin{eqnarray*}
\pi ^*\{\widetilde{C_{i,j}}\}&=&\{D_{i,j}\}-\{L_{i,j}\},\\
h_*\{D_{i,j}\}&=&0,\\
h_\{L_{i,j}\}&=&\{\widetilde{C_{3-i,3-j}}\}.
\end{eqnarray*}
For example, $h_*\{L_{0,1}\}=\widetilde{C_{2,3}}$. From this description,  by the proof of Lemma \ref{Lemma1} we have 
\begin{eqnarray*}
J_X^*(\alpha _2)\wedge J_X^*(\alpha _3)=J_X^*(\alpha _2\wedge \alpha _3) + \sum _{i,j} \{\alpha _2.\widetilde{C_{i,j}}\}.\{\alpha _3.\widetilde{C_{i,j}}\}[\widetilde{C_{3-i,3-j}}].
\end{eqnarray*}
Hence $J_X^*(\alpha _2)\wedge J_X^*(\alpha _3)=J_X^*(\alpha _2\wedge \alpha _3)$  iff for every $i,j$ we  have either $ \{\alpha _2.\widetilde{C_{i,j}}\}=0$ or $ \{\alpha _3.\widetilde{C_{i,j}}\}=0$. Equations of the form $ \{\alpha _2.\widetilde{C_{i,j}}\}=0$ and $ \{\alpha _3.\widetilde{C_{i,j}}\}=0$ are  linear and homogeneous in the cohomology classes of $\alpha _2$ and $\alpha _3$. Since we have $6$ such equations and the group $H^{1,1}(X)$ has dimension $5$,   we have that the solution set is non-empty and is a finite union of linear subspaces of $H^{1,1}(X)\times H^{1,1}(X)$, each of dimension at least $4$. 
 
 By solving these systems of linear equations, we find many such pairs $(\{\alpha _2\},\{\alpha _3\})$ which are nef classes. For example, $\{\alpha _2\}=a_2\eta _0+b_2(H-E_2-E_3)$ and $\{\alpha _3\}=a_3 \eta _0+b_3(H-E_1-E_3)$, with $a_2,b_2,a_3,b_3\geq 0$, are such pairs of nef classes.
 
 The assertion for the case $\alpha _2=\alpha _3$ is obvious. 
 
 4) Write $J_X^*(\alpha _2)\wedge J_X^*(\alpha _3) =J_X^*(\alpha _2\wedge \alpha _3)+\sum _{i,j}\{\alpha _2.\widetilde{C_{i,j}}\}\{\alpha _3.\widetilde{C_{i,jj}}\}[C_{3-i,3-j}]$. 
 
 By Theorem \ref{TheoremWeakBTCCondition}, we have in this case that the Lelong number of $f^*(\alpha _1)$ is zero at all curves $\widetilde{C_{3-i,3-j}}$ for which $\{\alpha _2.\widetilde{C_{i,j}}\}\{\alpha _3.\widetilde{C_{i,j}}\}\not= 0$. Then the proof of part 1 shows that we must have $\{\alpha _1.\widetilde{C_{i,j}}\}=0$ for such curves. 
 
 If $\alpha _1=\alpha _2=\alpha _3$, then it follows from the above that $f^*(\alpha _1)$, $f^*(\alpha _2)$ and $f^*(\alpha _3)$ satisfy the Weak BTC condition iff $\{\alpha _1\}$ satisfies the NIC condition, and hence $f^*(\alpha _1)$, $f^*(\alpha _2)$ and $f^*(\alpha _3)$ satisfy the BTC condition. Moreover, by Theorem \ref{TheoremWeakBTCCondition} we also have $\Lambda (f^*(\alpha _1),f^*(\alpha _2),f^*(\alpha _3))=f^*(\alpha _1)\wedge f^*(\alpha _2)\wedge f^*(\alpha _3)$. 
  
\subsection{Some generalisations} We can check that the proofs of the results proven so far in this paper apply to more general settings. 

In the key lemma, Lemma \ref{Lemma1}, it suffices to assume that $\alpha _1$ is smooth outside a variety $V_1$ and $\alpha _2$ is smooth outside a variety $V_2$, so that $V_1\cap V_2$ has dimension at most $1$. 

Therefore, to extend the results, we only need the condition so that $\lim _{n\rightarrow\infty}T_{1,n}\wedge f^*(\alpha _2\wedge \alpha _3)$ exists, for all good approximations $T_{1,n}$ of $T_1$. Analysing the proof of part 1 in Theorem \ref{TheoremSufficientCondition}, and using as mentioned previously that $f^*,f_*$ are bijective on the space of positive closed $(1,1)$ currents for pseudo-isomorphisms in dimension 3,  we see that it suffices to have that $\alpha _2\wedge \alpha _3$ can be written as the difference of two $(2,2)$ currents smooth outside some points, and that $f_*(T_1)$ is smooth outside these points. The readers can easily work out the details.   

\section{A discussion on intersection of non-positive closed $(1,1)$ currents}

So far, with some rare exceptions such as in the author's papers mentioned above, only the intersection of positive closed currents, whose resulting current is also positive, has been systematically discussed in the literature. However, even when we are only interested in the intersection of positive closed currents, the resulting current, even when defined, cannot always be positive, for example in the case where the intersection in cohomology is negative.   We also saw  that in the proof of Theorem \ref{TheoremMain}, it is natural and crucial to work with non-positive currents. Moreover, the intersection of non-positive currents is also interesting per se. In this section we will discuss generally the intersection of currents which are differences of positive closed $(1,1)$ currents. We see previously that there is the following dilemma,  discussed after the statement of Theorem \ref{TheoremBedfordTaylor}.

{\bf Dilemma}. When working with a non-positive closed $(1,1)$ current $T$, we see previously that there is a difficulty with choosing a pair of quasi-psh functions $u^+,u^-$ so that $T-(dd^cu^+-dd^cu^-)$ is a smooth closed $(1,1)$ form. If we have such a pair, then $u^++v,u^-+v$  is another pair, where $v$ is any quasi-psh function. If we choose $v$ to be too singular, then we may obtain not desired results. Moreover, apparently there may be other choices of pairs $\widetilde{u^+}, \widetilde{u^{-}}$  which are not of the above form, and so the question is how can we compare which pairs are better.

In the remaining of this section we will discuss two approaches to resolve this dillemma. 

\subsection{Smooth approximations which are linear}

It has been observed in \cite{truong1} that the construction of Dinh-Sibony can be applied simultaneously to all currents of every bidegrees instead of just one individual current in a fixed bidegree. More precisely, there is for each $n$ a linear operator $\mathcal{S}_n$ acting on the space of DSH currents so that for each $DHS$ current $T$, the sequence $\{\mathcal{S}_n(T)\}$ is a smooth approximation of $T$. Moreover, $\mathcal{S}_n$ has many good properties, for example it is compatible with the differentials $d$ and $d^c$; and if $T$ is continuous on an open set $U\subset X$, then $\mathcal{S}_n(T)$ are also continuous on the same set $U$ and converges locally uniformly to $T$ on $U$. The linearity of $\mathcal{S}_n$ solves the dilemma. In fact, it does not matter which representation we choose, since then we have by linearity of $\mathcal{S}_n$ and its compatibility with $d$ and $d^c$ that: 
$$dd^c(\mathcal{S}_n(u_1^+)-\mathcal{S}_n(u_1^-))=\mathcal{S}_n(dd^cu_1^+-dd^cu_1^-)=\mathcal{S}_n(T_1-\Omega _1)$$   
is independent of the choice of the representation. Moreover, since $\mathcal{S}_n(\Omega _1)$ converges uniformly to $\Omega _1$, we can replace the sequence $\{\Omega _1+dd^c(\mathcal{S}_n(u_1^+)-\mathcal{S}_n(u_1^-))\}$ by the sequence $\{\mathcal{S}_n(T_1)\}$.

While this is good, it seems that this choice of good approximations may be too restrictive. One drawback is that it is unknown whether we can find such $\mathcal{S}_n$ which when applied to a quasi-psh function will produce a decreasing sequence of quasi-psh functions. Likewise, it is unknown whether there are such $\mathcal{S}_n$ which will give us Demailly's Approximation Theorem 2, which is crucial in proving some positivity of intersection (sub)measures. Our speculation is that such $\mathcal{S}_n$ do not exist.

\subsection{Representation with minimal singularities}

 The previous subsection provides a way to resolve this dilemma, but on the other hand it is unknown whether that solution will provide  Demailly's Approximations Theorems 1 and 2, which we would like to incorporate. 

An intersection theory of non-positive $(1,1)$ currents should include as a special case the intersection theory of positive closed $(1,1)$ currents. Let $T$ be a positive closed $(1,1)$ current,  then we can write $T=\Omega +dd^cu$, where $\Omega $ is smooth and $u$ is quasi-psh. Of course, we can also write more generally $T=\Omega +dd^cu^+-dd^cu^{-}$, where $u^{\pm}$ are quasi-psh functions. However, the representation $T=\Omega +dd^cu=\Omega +dd^cu-dd^c0$ is canonical in a number of ways. First, if we have two representations $T=\Omega +dd^c u=\widetilde{\Omega} +dd^c\widetilde{u}$ then $u-\widetilde{u}$ is a smooth function, and hence they have the same singularity.  Second, in any other representation $T=\Omega +dd^cu^+-dd^cu^{-}$, we have  that up to an additive constant $u^{+}=u+u^{-}$, and hence the singularity of $u$ will not be more than that of $u^{+}$. Trivially, we also have that the singularity of the function $0$ is not more than that of $u^-$. Therefore, we can say that the pair $(u,0)$ has the least singularity among all the representations $(u^+,u^-)$ of $T$. 

Now, when we work with an arbitrary current $T$ which is a difference of two positive closed $(1,1)$ currents, we extend the above observation. So we say that a pair $(u^+,u^-)$ is a canonical representation of $T$ if it has the least singularity among such pairs. Here is an idea of how to find such canonical representations. We fix the closed smooth $(1,1)$ form $\Omega$ in the representations $[T,\Omega ,u^+,u^-]$. We fix a  K\"ahler form $\omega$ on $X$. By this normalisation, for any representation $[T,\Omega ,u^+,u^-]$ the function $\phi :=u^+-u^-\in [-\infty , \infty ]$ is, up to an additive constant, independent of the representation. We note that the set $I(\phi ):=\{x\in X:~\phi (x)=-\infty\}\cup \{x\in X:~\phi (x)=\infty\}=\{x\in X:~u^+ (x)=-\infty\}\cup \{x\in X:~u^- (x)=-\infty\}$ is pluripolar, and hence does not affect whether a function is quasi-psh or not. 

Now we fix a positive closed smooth $(1,1)$ form $\omega $ on $X$. For each $n\in \mathbb{N}$, we denote by $\mathcal{R}_{\omega ,n}(T,\Omega )$ the  set of all representations $[T,\Omega ,u^+,u^-]$ so that both $u^{\pm}$ are $n\omega $-quasi psh (that is $n\omega +dd^cu^{\pm}\geq 0$). We note that $\mathcal{R}_{\omega ,n}(T,\Omega )$ is non-empty provided that $n$ is large enough. We have the following result about representations with minimal singularities. 

\begin{lemma}
There exists one $[T,\Omega ,u^+_{\omega ,n},u^{-}_{\omega ,n}]\in \mathcal{R}_{\omega ,n}(T,\Omega )$ so that for all $[T,\Omega ,u^+,u^-]\in  \mathcal{R}_{\omega ,n}(T,\Omega )$, the pair $(u^+_{\omega ,n},u^{-}_{\omega ,n})$ is not singular than the pair $(u^+,u^-)$.
 \label{LemmaMinimalSingularities}\end{lemma}
 \begin{proof}
 We can, up to an additive constant, identify $\mathcal{R}_{\omega ,n}(T,\Omega )$ with the set $qPSH(n\omega ,\phi)=\{u:~$ $u$ and $u+\phi$ are both $n\omega$-quasi psh, $u\leq 0\}$. In fact, we only need to define $u=u^-$. Then the regularisation $u^{-}_{\omega ,n}$ of the function $\sup _{u\in qPSH(n\omega ,\phi)}$ is also in $qPSH(n\omega ,\phi)$ (note that $\sup _u(u+\phi )=\phi +\sup _uu$ on $X\backslash I(\phi )$). Then we define $u^+_{\omega ,n}=u^-_{\omega ,n}+\phi $, with appropriate modification on the set $I(\phi )$. 
\end{proof}
 
By definition, we see that $(u^+_{\omega ,n},u^{-}_{\omega ,n})$ is less singular when $n$ increases. So somehow the "limit" of these minimal pairs give us the "canonical representation" of $T$. We now apply this to the question concerned in this paper: to define a Bedford-Taylor's type monotone convergence for $T_1$ and $T_2$,  $T_3$ differences of positive closed $(1,1)$ currents on a compact K\"ahler manifold $X$ of dimension $3$, so that $T_2\wedge T_3$ is classically defined.  Fix $\Omega _1$ a smooth closed $(1,1)$ form cohomologous to $T_1$. For each $n\in \mathbb{N}$, we define $\mathcal{G}(T_1,T_2, T_3,\Omega _1,n)$ to be the set of cluster points of sequences of signed measures $\{(\Omega _1+dd^c u^+_{\omega ,n,k}-dd^c u^-_{\omega ,n,k})\wedge T_2\wedge T_3\}_{k=1,2,\ldots }$, where $\{(u^+_{\omega ,n,k},u^-_{\omega ,n,k})\}_{k}$ is a good approximation of $(u^+_{\omega ,n},u^{-}_{\omega ,n})$. We then define the following Bedford-Taylor's type monotone convergence. 

\begin{definition} Let $X$ be a compact K\"ahler manifold of dimension $3$. Let $T_1$, $T_2$ and $T_3$ be differences of positive closed $(1,1)$ currents on $X$. Assume that $T_2\wedge T_3$ is classically defined. 

{\bf Weak BTC-B condition.}  (For weak Bedford-Taylor's type monotone convergence, version B) $T_1,T_2,T_3$ satisfy Weak BTC-B condition if there is a smooth closed $(1,1)$ form $\Omega _1$ cohomologous to $T_1$, a positive integer $n\geq 0$  and a signed measure $\mu _{T_1, T_2,T_3,\Omega _1,n}$ so that $\mathcal{G}(T_1,T_2, T_3,\Omega _1,n)=\{\mu _{T_1, T_2,T_3,\Omega _1,n}\}$. 

{\bf BTC-B condition} (For Bedford-Taylor's type monotone convergence, version B) $T_1,T_2,T_3$ satisfy  BTC-B condition if there are a signed measure $\mu _{T_1,T_2,T_3}$ and an integer $n_0$ so that for all smooth closed $(1,1)$ forms $\Omega _1$ cohomologous to $T_1$ and for all $n\geq n_0$, we have $\mathcal{G}(T_1,T_2, T_3,\Omega _1,n)=\{\mu _{T_1, T_2,T_3}\}$. 
\label{DefinitionBedfordTaylorTypeB}\end{definition}

We can also modify the above definitions, for example, to require that $\lim _{n\rightarrow\infty}\mathcal{G}(T_1,T_2, T_3,\Omega _1,n)$ exists and is a  signed measure  for some choice of $\Omega _1$ in the Weak BTC-B condition, and $\lim _{n\rightarrow\infty}\mathcal{G}(T_1,T_2, T_3,\Omega _1,n)$ exists and is a signed measure for all $\Omega _1$ in the BTC-B condition.

\end{document}